\documentclass{article}
\usepackage{xcolor}
\usepackage{pagecolor}
\definecolor{black}{cmyk}{1 1 1 1}
\pagecolor{white}
\color{black}
\RequirePackage[auth-sc-lg,affil-it]{authblk}
\RequirePackage{etoolbox}
\RequirePackage{hyperref}
\hypersetup{
    colorlinks=true,
    linkcolor=black,
    filecolor=black,      
    urlcolor=blue,
    citecolor=black,
}
\makeatletter
\patchcmd{\@maketitle}{center}{flushleft}{}{}
\patchcmd{\@maketitle}{center}{flushleft}{}{}
\patchcmd{\@maketitle}{\LARGE}{\LARGE\bf\boldmath}{}{}

\def\maketitle{{%
  \renewenvironment{tabular}[2][]
    {\begin{flushleft}}
    {\end{flushleft}}
  \AB@maketitle}}
\makeatother

\newcommand{\email}[1]{{\rm\href{mailto:#1}{#1}}}
\renewenvironment{abstract}
  {\noindent\textbf{Abstract:}\quad}
  {}
\newcommand\keywords[1]{%
  \let\svthefootnote\thefootnote
  \let\thefootnote\relax\footnotetext{\textbf{Key words and phrases:} #1}%
  \let\thefootnote\svthefootnote%
}
\newcommand\classification[1]{%
  \let\svthefootnote\thefootnote
  \let\thefootnote\relax\footnotetext{\textbf{AMS (MOS) Subject Classifications:} #1}%
  \let\thefootnote\svthefootnote%
}

\RequirePackage[text={4.5in,7.125in},centering,paperwidth=6.0in,paperheight=9.0in]{geometry}
\usepackage[T1]{fontenc}
\usepackage{enumitem}
\usepackage{amssymb}
\usepackage{amsmath}
\usepackage{amsthm}
\usepackage{mathrsfs}
\usepackage{fullpage}
\usepackage{scrextend}
\usepackage{xstring}
\usepackage{tikz,pgf}
\usetikzlibrary{calc,positioning}
\usetikzlibrary{decorations.markings}
\usetikzlibrary{fadings,decorations.pathreplacing}
\usetikzlibrary{snakes}
\usetikzlibrary{through}
\usetikzlibrary{shapes}
\theoremstyle{definition}

\newtheorem{conjecture*}{Conjecture}
\newtheorem{construction}{Construction}[section]

\theoremstyle{plain}

\newtheorem{theorem}[construction]{Theorem}
\newtheorem{theorem*}{Theorem}
\newtheorem{corollary}[construction]{Corollary}

\theoremstyle{remark}

\newcommand{\GCD}[1]
{
 \IfEq{#1}{}
  {\texttt{gcd}}
  {\texttt{gcd}\ensuremath{\left(#1\right)}}
}
\newcommand{\PSTS}[1]
{
 \IfEq{#1}{}
  {\textsc{psts}}
  {\textsc{psts}\ensuremath{\left(#1\right)}}
}

\newcommand{\STS}[1]
{
 \IfEq{#1}{}
  {\textsc{sts}}
  {\textsc{sts}\ensuremath{\left(#1\right)}}
}

\newcommand{\RGDD}[2]
{
 \IfEq{#1}{}
  {\textsc{rgdd}}
  {\ensuremath{#1}\text{-}\textsc{rgdd}}
 \IfEq{#2}{}
  {}
  {\text{ of type }\ensuremath{#2}}
}

\newcommand{\IGDD}[2]
{
 \IfEq{#1}{}
  {\textsc{igdd}}
  {\ensuremath{#1}\text{-}\textsc{igdd}}
 \IfEq{#2}{}
  {}
  {\text{ of type }\ensuremath{#2}}
}

\newcommand{\IGDDS}[2]
{
 \IfEq{#1}{}
  {\textsc{igdd}s}
  {\ensuremath{#1}\text{-}\textsc{igdd}s}
 \IfEq{#2}{}
  {}
  {\text{ of type }\ensuremath{#2}}
}

\newcommand{\GDD}[2]
{
 \IfEq{#1}{}
  {\textsc{gdd}}
  {\ensuremath{#1}\text{-}\textsc{gdd}}
 \IfEq{#2}{}
  {}
  {\text{ of type }\ensuremath{#2}}
}

\newcommand{\GDDS}[2]
{
 \IfEq{#1}{}
  {\textsc{gdd}s}
  {\ensuremath{#1}\text{-}\textsc{gdd}s}
 \IfEq{#2}{}
  {}
  {\text{ of type }\ensuremath{#2}}
}
\newcommand{\PBD}[2]
{
 \IfEq{#1}{}
  {\textsc{pbd}}
  {\ensuremath{#1}\text{-}\textsc{pbd}}
 \IfEq{#2}{}
  {}
  {\ensuremath{\left(#2\right)}}
}
\newcommand{\TD}[1]
{
 \IfEq{#1}{}
  {\textsc{td}}
  {\textsc{td}\ensuremath{\left(#1\right)}}
}
\newcommand{\OA}[1]
{
 \IfEq{#1}{}
  {\textsc{OA}}
  {\textsc{OA}\ensuremath{\left(#1\right)}}
}

\newcommand{\MOLS}[1]
{
 \IfEq{#1}{}
  {\textsc{mols}}
  {\textsc{mols}\ensuremath{\left(#1\right)}}
}

\newcommand{\Field}[1]
{
 \IfEq{#1}{}
  {\mathbb{F}}
  {\mathbb{F}\ensuremath{_{#1}}}
}
\newcommand{\Integers}[1]
{
 \IfEq{#1}{}
  {\mathbb{Z}}
  {\mathbb{Z}\ensuremath{_{#1}}}
}

\title{Nonsequenceable Steiner triple systems}

\author[1]{Donald L. Kreher}
\author[2]{Douglas R. Stinson}
\affil[1]{Michigan Technological University%
\newline Houghton, Michigan 49931 U.S.A.%
\newline\email{kreher@mtu.edu}
\bigskip
}
\affil[2]{David R. Cheriton School of Computer Science%
\newline University of Waterloo%
\newline Waterloo, Ontario N2L 3G1, Canada%
\newline\email{dstinson@uwaterloo.ca}
}

\date{}

\begin{document}
\maketitle
\begin{abstract}
A partial Steiner triple system is 
is \emph{sequenceable} if the points can be sequenced so
that no proper segment can be partitioned into blocks.
We show that, 
if $0 \leq a \leq (n-1)/3$, then
there exists a nonsequenceable $\PSTS{n}$
of size $\frac{1}{3}\binom{n}{2}-a$, 
for all $n \equiv 1 \pmod{6}$ except for $n=7$.
\end{abstract}

\section{Introduction}
A decomposition of the complete graph on $n$ points into 
triangles is called a \emph{Steiner triple system} 
of order $n$ and is denoted by $\STS{n}$. The vertex sets of 
the triples used are called the \emph{blocks} of the Steiner 
triple system. Thus, equivalently, an $\STS{n}$ is a pair
$(X,B)$, where $X$ is an $n$-element set of \emph{points}
and $B$ is a collection of $3$-element subsets of $X$
called \emph{blocks}, such that every pair of points is contained 
in exactly one block. It is well known that an $\STS{n}$ 
exists if and only if $n \equiv 1,3 \pmod{6}$. A decomposition of a proper subgraph of the complete graph on $n$ points into
triangles is called a \emph{partial Steiner triple system}
of order $n$ and is denoted by $\PSTS{n}$. The \emph{size} of a $\PSTS{n}$ is the number of blocks it contains.

An $\STS{n}$ is \emph{sequenceable} if the points can be sequenced so 
that no proper segment can be partitioned into blocks.
Such a sequence is called an \emph{admissible sequence}.
If  an $\STS{n}$ has no admissible sequence, then we say it is
\emph{nonsequenceable}.  
For example, $\{\mathit{abd}, \mathit{bce}, \mathit{cdf}, \mathit{deg}, \mathit{efa}, \mathit{fgb}, \mathit{gac}\}$,
the unique (up to isomorphism) $\STS{7}$ has the admissible
sequence $\mathit{abcdefg}$. 
A fascinating study of sequenceable partial 
Steiner triple systems can be found in~\cite{man41}.

In \cite{KS} a related problem is examined. In this article the authors ask if the vertices of a Steiner triple system can be sequenced such that no length $\ell$ segment of the sequence contains a block. They obtain results for $\ell=3$ and $4$.

A set of disjoint blocks in an $\STS{n}$ is called a 
\emph{partial parallel class}.
Clearly any partial parallel class contains at most 
$\lfloor v/3 \rfloor$ blocks. A partial parallel class 
containing all the points of a design is called a \emph{parallel 
class} and a partial parallel class containing all but one 
of the points of a design is called an \emph{almost parallel class}. 

\begin{theorem}\label{one}
Suppose a  $\PSTS{n}$ has the property, for $n-1$ distinct points $x$, that 
there is an almost parallel class
that does not contain $x$. Then the $\PSTS{n}$
 is nonsequenceable.
\end{theorem}

\begin{proof}
Consider any sequence $x_1x_2\cdots x_n$ of the vertices 
of such a $\PSTS{n}$.   There is either an almost parallel  class that 
does not contain $x_1$ or one that does not contain $x_n$.  
Thus the blocks of the first almost  parallel class, if it exists, would
partition the segment $x_2\cdots x_{n-1}x_n$,
and the latter would partition
the segment $x_1x_2\cdots x_{n-1}$.
\end{proof}

\section{Example applications of Theorem~\ref{one}}\label{examples}
\begin{description}[leftmargin=1em]
\item[\boldmath\STS{13}:]
Developing the base blocks

\hfill\begin{minipage}{.875\textwidth}\begin{flushleft}
$\{0,2,7\}$,
$\{0,1,4\}$,
\end{flushleft}\end{minipage}

modulo 13 generates an $\STS{13}$ that contains the almost parallel class 

\hfill\begin{minipage}{.875\textwidth}\begin{flushleft}
\noindent%
\{7,8,11\}$,
\{3,5,10\}$,
\{2,4,9\}$,
\{1,6,12\}$.
\end{flushleft}\end{minipage}

Because any translate of this almost parallel class is again an almost parallel class, there is an almost parallel class missing any desired point.
Thus, by Theorem~\ref{one} this $\STS{13}$ is nonsequenceable.

\item[\boldmath\STS{19}:]
Developing the base blocks 

\hfill\begin{minipage}{.875\textwidth}\begin{flushleft}
$\{0,1,6\}$,
$\{0,2,10\}$,
$\{0,3,7\}$,
\end{flushleft}\end{minipage}

modulo 19 generates an $\STS{19}$ that contains the almost parallel class 

\hfill\begin{minipage}{.875\textwidth}\begin{flushleft}
\noindent%
$\{1,3,11\}$,
$\{2,15,16\}$,
$\{4,17,18\}$,
$\{5,8,12\}$,
$\{6,9,13\}$,
$\{7,10,14\}$.
\end{flushleft}\end{minipage}

Thus by Theorem~\ref{one} this $\STS{19}$ is nonsequenceable.

\item[\boldmath\STS{25}:]
Developing the base blocks

\hfill\begin{minipage}{.875\textwidth}\begin{flushleft}
\noindent%
$\big\{(0,0),(0,1),(2,3)\big\}$,
$\big\{(0,0),(1,2),(2,0)\big\}$,
$\big\{(0,0),(1,0),(3,1)\big\}$,
\end{flushleft}\end{minipage}

modulo 5 independently in both coordinates generates a
$\STS{25}$ with vertices $\Integers{5}\times\Integers{5}$
that contains the almost parallel class

\hfill\begin{minipage}{.875\textwidth}\begin{flushleft}
\noindent%
$\big\{(0,1),(0,2),(2,4)\big\}$,
$\big\{(1,0),(3,2),(1,4)\big\}$,
$\big\{(1,1),(4,1),(0,3)\big\}$,
$\big\{(2,0),(2,3),(3,4)\big\}$,
$\big\{(2,1),(2,2),(4,4)\big\}$,
$\big\{(3,0),(1,2),(1,3)\big\}$,
$\big\{(3,1),(4,2),(3,3)\big\}$,
$\big\{(4,0),(4,3),(0,4)\big\}$.
\end{flushleft}\end{minipage}

It follows from Theorem~\ref{one} that
this $\STS{25}$ is nonsequenceable.
\item[\boldmath\STS{31}:]
Developing the base blocks 

\hfill\begin{minipage}{.875\textwidth}\begin{flushleft}
\noindent%
$\{0,5,11\}$,
$\{0,4,12\}$,
$\{0,3,13\}$,
$\{0,2,9\}$,
$\{0,1,15\}$,
\end{flushleft}\end{minipage}

modulo 31 generates an $\STS{31}$ that contains the almost parallel class 

\hfill\begin{minipage}{.875\textwidth}\begin{flushleft}
\noindent%
$\{11,15,23\}$,
$\{10,26,27\}$,
$\{9,14,20\}$,
$\{8,13,19\}$,
$\{7,25,28\}$,
$\{5,21,22\}$,
$\{4,24,29\}$,
$\{3,6,16\}$,
$\{2,12,30\}$,
$\{1,17,18\}$.
\end{flushleft}\end{minipage}

Thus by Theorem~\ref{one} this $\STS{31}$ is nonsequenceable.
\item[\boldmath\STS{43}:]
Developing the base blocks

\hfill\begin{minipage}{.875\textwidth}\begin{flushleft}
\noindent%
$\{0,1,16\}$,
$\{0,2,14\}$,
$\{0,3,11\}$,
$\{0,4,37\}$,
$\{0,5,25\}$,
$\{0,7,24\}$,
$\{0,9,22\}$,
\end{flushleft}\end{minipage}

modulo 43 generates an $\STS{43}$ that contains the almost 
parallel class 

\hfill\begin{minipage}{.875\textwidth}\begin{flushleft}
\noindent$\{14,15,30\}$,
$\{1,28,29\}$,
$\{2,20,25\}$,
$\{3,23,41\}$,
$\{4,33,35\}$,
$\{5,34,36\}$,
$\{6,19,40\}$,
$\{7,39,42\}$,
$\{8,26,31\}$,
$\{9,27,32\}$,
$\{10,37,38\}$,
$\{11,17,21\}$,
$\{12,18,22\}$,
$\{13,16,24\}$.
\end{flushleft}\end{minipage}

Thus by Theorem~\ref{one} this $\STS{43}$ is nonsequenceable.

\end{description}

\section{Constructions}

A group divisible design with blocks of size 3, having 
$u$ groups of size $g$ and $v$ groups of size $m$ (denoted by $\GDD{3}{g^um^v}$),
is a decomposition of the complete multipartite graph
\[
 K_{\underbrace{g,g,\ldots,g}_u,\underbrace{m,m,\ldots,m}_v}
\]
into triangles. The triangles (or triples ) used in the decomposition are the blocks of the $\textsc{gdd}$.
If $6\mid m$ and $6 \mid g$, then a $\GDD{3}{m^u}$ and
a $\GDD{3}{g^um^1}$ exist for all $u\geq 3$  
\cite{HB555,HB2211}.

\begin{theorem}\label{A}
There exists a nonsequenceable $\STS{n}$ for all $n  \equiv  1 \pmod{6}$ except for $n=7$.
\end{theorem}
\begin{proof}
First suppose $n = 1+6k$ with $k=2u$.
If $u=1$, then $n=13$ and if $u=2$, then $n=25$.  
For both orders $n=13$ or $25$, there is an example of a nonsequenceable $\STS{n}$ given in Section~\ref{examples}. 

If $u \geq 3$, there exist a $\GDD{3}{12^u}$ with groups 
(partite sets) $G_1,G_2,\ldots , G_u$.
Let $X$ be a new point.
In each group $G_i$, fill in 
$G_i \cup \{X\}$
with the  nonsequenceable
$\STS{13}$  found in Section~\ref{examples}.

We now show this $\textsc{STS}$ satisfies the the conditions of Theorem~\ref{one}.
Clearly it is satisfied for the point X. 
For any other point, say $y \in G_i$, take the almost
parallel class in the $i$-th $\STS{13}$ that misses $y$, 
and for all $j \neq i$,
take the almost parallel class in the $j$-th $\STS{13}$ 
that misses $X$. The union
of these partial parallel classes is an almost parallel class that misses $y$.

Now suppose $n = 1+6k$ with $k=2u+3$.
If $u = -1$, then $n=7$  and, as noted in the Theorem, a
nonsequenceable $\STS{7}$ does not exist. 
If $n=0,1$ or $2$, then $n=19,31$ or $43$ respectively.
There is an example of a nonsequenceable $\STS{n}$ given in of Section~\ref{examples} for each $n \in \{19,31,43\}$.
If $u \geq 3$, there exists
a $\GDD{3}{12^u18^1}$ with groups (partite sets)
$G_1,G_2,\ldots ,G_u,M$, where $|G_i|=12$ for all $i$ and $|M|=18$.
Let $X$ be a new point.
In each group $G_i$, fill in 
$G_i \cup \{X\}$ 
with the  nonsequenceable  $\STS{13}$ 
given in Section~\ref{examples}. On $M \cup \{X\}$, place the nonsequenceable  $\STS{19}$ found in Section~\ref{examples}. The rest of the proof is the same as 
for $k$ even, and we leave it for the reader to check the details.
\end{proof}

\begin{corollary}
If $0 \leq a \leq (n-1)/3$, then
there exists a nonsequenceable $\PSTS{n}$ 
of size $\frac{1}{3}\binom{n}{2}-a$,
for all $n  \equiv  1 \pmod{6}$ except for $n=7$.
\end{corollary}
\begin{proof}
The $\STS{n}$ constructed in Theorem~\ref{A} has, for each point $x$, an almost parallel class that does not contain $x$. Fixing any point $x_0$
and removing $a$ of the blocks in the almost parallel class that does
not contain $x_0$  constructs a $\PSTS{n}$
of size $\frac{1}{3}\binom{n}{2}-a$ that still satisfies 
Theorem~\ref{one}.
\end{proof}

\section*{Closing remarks.}
Brian Alspach gave a talk on August 9, 2018 entitled ``Strongly Sequenceable 
Groups'' at the 4th Kliakhandler Conference held at MTU. In this 
talk, among other things, the notion of sequencing diffuse posets was introduced and the following research problem was posed:
\begin{quotation}
``Given a triple system of order $n$ with $\lambda=1$, define a poset $P$ by letting its elements be the triples and any union of disjoint triples. This poset is not diffuse in general, but it is certainly possible that $P$ is sequenceable.''
\end{quotation}
Our article shows surprisingly that there are in fact triple systems where 
the poset $P$ is nonsequenceable. However, it still remains an interesting open question to construct a $\STS{n}$ that is nonsequenceable when $n \equiv 3 \pmod{6}$.
A more ambitious research problem would be to determine necessary and sufficient conditions for when an $\STS{n}$ is sequenceable.


\begin{thebibliography}{99}
\bibitem{man41}
B. Alspach, D.L. Kreher, and A.G. Pastine,
Sequencing Partial Steiner Triple Systems,
\emph{preprint}.
\bibitem{HB555}
C.J. Colbourn, D.G. Hoffman, and R.S. Rees, 
A new class of group divisible designs with block size three, 
\textsl{J. Combin. Theory A} \textbf{59} (1992) 73--89.
\bibitem{KS}
D.L. Kreher and D.R. Stinson
Block-avoiding sequencings of points in Steiner triple systems,
\emph{preprint}.
%\textsl{Australas. J. Combin.} 
\bibitem{HB2211}
L. Zhu,
Some recent developments on BIBDs and related designs, 
\textsl{Discrete Math.} \textbf{123} (1993) 189--214.
\end{thebibliography}
\end{document}